\documentclass[12pt]{amsart}

\usepackage{hyperref}
\usepackage{graphicx}
\usepackage[ansinew]{inputenc}
\usepackage{enumerate}
\usepackage{amsmath}
\usepackage{amsfonts}
\usepackage{amssymb}
\usepackage{amsthm}
\usepackage{multirow}
\usepackage{graphicx}
\usepackage{color}
\usepackage{microtype}
\input xy
\xyoption{all}
\usepackage{url}
\usepackage{varioref}
\usepackage{eso-pic}
\usepackage{tikz}
\usepackage{float}
\usepackage{dsfont}
\usepackage{listings}

\lstdefinelanguage{Magma}
{
keywords={for,end,if,then,else,elif,while,function,return,cat,&,and,or },
morekeywords={Seqset,Setseq,Polytope,AutomorphismGroup,RowSequence,IdentifyGroup,
	      Subgroups,PermutationMatrix,Generators,MatrixGroup,Transpose},
sensitive=false,
morecomment=[l]{//},
morecomment=[s]{/*}{*/},
morestring=[b]",
}

\lstnewenvironment{codice_magma}[1][]
{\lstset{basicstyle=\small\ttfamily, columns=fullflexible,
language=Magma,
keywordstyle=\color{red}\bfseries,
commentstyle=\color{blue},
numbers=left, numberstyle=\tiny,
stepnumber=2, numbersep=5pt, frame=single, #1}}{}

\newtheorem{thm}{Theorem}
\newtheorem*{thm*}{Theorem}
\newtheorem{lem}[thm]{Lemma}
\newtheorem{prop}[thm]{Proposition}
\newtheorem{coroll}[thm]{Corollary}

\theoremstyle{definition}
\newtheorem{defin}{Definition}
\newtheorem{question}{Question}

\theoremstyle{remark}
\newtheorem{remark}[thm]{Remark}

\newcommand{\M}{\mathcal{M}}

\newcommand{\C}{\mathbb{C}}

\newcommand{\N}{\mathbb{N}}
\newcommand{\Z}{\mathbb{Z}}

\newcommand{\HH}{H}
\newcommand{\E}{{\mathcal E}}
\newcommand{\EE}{{\mathfrak E}}
\newcommand{\I}{{\mathcal I}}
\newcommand{\J}{{\mathcal J}}
\newcommand{\OO}{{\mathcal O}}
\newcommand{\ux}{\underline{x}}
\DeclareMathOperator{\PP}{\mathbb{P}}

\newcommand{\cL}{\mathcal{L}}
\newcommand{\cI}{\mathcal{I}}

\DeclareMathOperator{\id}{Id}

\DeclareMathOperator{\h}{h}

\DeclareMathOperator{\Tors}{Tors}

\DeclareMathOperator{\Aut}{Aut}

\DeclareMathOperator{\Pic}{Pic}
\DeclareMathOperator{\PGL}{PGL}

\DeclareMathOperator{\Fix}{Fix}

\newcommand{\ti}[1]{\tilde{#1}}

\newcommand{\spa}[1]{\left < #1 \right >}

\title[A twisted bicanonical system with base points]{A twisted bicanonical system with base points}

\thanks{The first author was partially supported by CIRM "Centro Internazionale per la Ricerca Matematica" (Fondazione Bruno Kessler). The second author was partially supported by the projects PRIN2010-2011 ''Geometria delle variet\`a algebriche" and Futuro in Ricerca 2012 "Moduli Spaces and Applications". Both authors are members of GNSAGA-INdAM. The second author is indebted with Jorge Neves for some inspiring discussions we had during the preparation of \cite{BFNP13}, that originated some of the results in section \ref{twistedbicanonical}. Both authors thank the anonymous referee for correcting a mistake in Theorem \ref{thm: 3}.
}

\title{A twisted bicanonical system with base points}
\author{Filippo F. Favale}
\address[Filippo F. Favale]{Department of Mathematics, University of Trento, via Sommarive 14,
I-38123 Trento, Italy}
\email{filippo.favale@unitn.it}

\author{Roberto Pignatelli}
\address[Roberto Pignatelli]{Department of Mathematics, University of Trento, via Sommarive 14,
I-38123 Trento, Italy}
\email{Roberto.Pignatelli@unitn.it}


\begin{document}

\maketitle

\begin{center}
\textit{Dedicated to the memory of Alexandru Lascu}
\end{center}

\begin{abstract}
By a theorem of Reider, a twisted bicanonical system, that means a linear system of divisors numerically equivalent to a bicanonical divisor, on a minimal surface of general type, is base point free if $K^2_S \geq 5$. Twisted bicanonical systems with base points are known in literature only for $K^2=1,2$. 

\noindent We prove in this paper that all surfaces in a family of surfaces with $K^2=3$ constructed in a previous paper with G. Bini and J. Neves have a twisted bicanonical system (different from the bicanonical system) with two base points.  

\noindent We show that the map induced by the above twisted bicanonical system is birational, and describe in detail the closure of its image and its singular locus. Inspired by this description, we reduce the problem of constructing a minimal surface of general type with $K^2=3$ whose bicanonical system has base points, under some reasonable assumptions, to the problem of constructing a curve in $\mathbb{P}^3$ with certain properties.
\end{abstract}

\section{Introduction}

The study of the $n-$canonical systems $|nK|$ and of the induced $n-$canonical maps $\varphi_{|nK|}$ is one of the  main fields of the theory of surfaces of general type of the last four decades, from the celebrated results of Bombieri \cite{bombieri} (recently extended to higher dimension, see {\it e.g.} \cite{hacon}) proving, among other things, that $\varphi_{|nK|}$ is an embedding from $n\geq 5$, and has no base points for $n \geq 4$, or $n=2,3$ with some exceptions.

One of the principal field of study originated by Bombieri's paper is the analysis of the exceptions, the minimal surfaces of general type whose  $n-$canonical system has some special behaviour. At the end of the 80s Reider improved Bombieri's result.  We report here only the part of the theorem related to the base point freeness of the bicanonical system.

\begin{thm*}[\cite{reider}]\label{reider}
Let $S$ be a minimal surface of general type. If $K^2_S \geq 5$, then $|2K_S|$ has no base points.
\end{thm*}

This theorem has been improved by the contribution of several authors (see \cite[Theorem 3.8]{survey} and the references therein). Finally \cite{cirofab} completed the proof that $p_g(S)>0$ implies that $|2K_S|$ has no base points. So a minimal surface of general type whose bicanonical system has base points has $p_g=0$ and $K_S^2\in \{1,2,3,4\}$. 

If $K_S^2=1$, the bicanonical system is a pencil with self intersection $4$ and therefore has always base points. A bit more is true: \cite[Theorem 5.1]{pencils} proves that in this case the bicanonical system can't be a morphism. 

The only other known examples of minimal surfaces of general type whose bicanonical system has base points have $K_S^2=2$ and fundamental group either $\Z_3 \times \Z_3$ or $\Z_9$. These families have been constructed and studied in  \cite{mendes lopes pardini} (although the first family is due to Beauville and Xiao Gang, see \cite{beauville}). For all these surfaces the bicanonical system has exactly two base points. Moreover  \cite{mendes lopes pardini} describes the whole connected components of the moduli space of surfaces of general type containing these families, showing that if the fundamental group is cyclic all the surfaces in the component have bicanonical system with two base points, whereas in the other case we can deform them to surfaces with free bicanonical system. More precisely, if we deform the surface to remove one of the base points, we automatically lose also the other.

We think then that the two major questions left in this field are the following.
 \begin{question}\label{question 1}
 Do there exist minimal surfaces of general type whose bicanonical system has base points and $K_S^2 = 3$ resp. $4$?
 \end{question}
 \begin{question}\label{question 2}
Do there exist minimal surfaces of general type whose bicanonical system has an odd number of base points?
\end{question}

Linear systems of divisors which are numerically equivalent to canonical divisors have been largely studied: see for example \cite{bea88}, \cite{gl}, \cite{pirola1}, \cite{pirola2}.
We will call such linear systems {\it twisted canonical systems} although we stress that, in the literature, one often find also the terminology ``paracanonical systems''.
In this article we are interested in linear systems spanned by divisors which are numerically equivalent to a bicanonical divisor, i.e. in {\it twisted bicanonical systems}. 
\vspace{2mm}


Notice that Reider's main theorem in \cite{reider} is a result on the properties of adjoint linear systems $|K+L|$, whose assumptions on $L$ are purely numerical. In particular, Reider's proof shows 
\begin{thm*}[Reider]
If $S$ is a minimal surface of general type with $K^2_S \geq 5$, then all twisted bicanonical systems have no base points.
\end{thm*}

It follows that the questions above deserve to be investigated also for the slightly larger class of twisted bicanonical systems. In particular
\begin{question}\label{question 3}
 Do there exist minimal surfaces of general type with a twisted bicanonical system having base points and $K_S^2 = 3$ resp. $4$?
 \end{question}
 \begin{question}\label{question 4}
Do there exist minimal surfaces of general type with a twisted bicanonical system having an odd number of base points?
\end{question}

Our first result is that the answer to Question \ref{question 3} is positive in the case $K^2=3$.

\begin{thm}
\label{thm: 1}
Let $\M$ be the irreducible family of minimal surfaces of general type with $p_g=0$, $K^2=3$ and fundamental group $\Z_4 \ltimes \Z_4$ constructed in \cite{BFNP13}. Then there is a $2-$torsion $\eta \in \Pic(S)$ not divisible by $2$ such that $|2K+\eta|$ has exactly two simple base points. 
\end{thm}

Note that the base points are two, so Question \ref{question 4} remains completely open as well as the following one

 \begin{question}\label{question 5}
What can be said about the base point of twisted pluricanonical systems for minimal surfaces of general type with $p_g>0$?
\end{question}

By \cite[Theorem 5.6]{BFNP13}, $\M$ is an open subset of an irreducible component of the moduli space of surfaces of general type, and therefore in this case the base points can't be removed by a small deformation. 

A surface $S$ in $\M$ has $\Pic(S) \cong H_1(S,\Z) \cong \Z/2\Z \times \Z/4\Z$, and therefore $8$ distinct twisted bicanonical systems, including the bicanonical system. We have explicitly written all of them and computed that all the other $7$ have no base points. Therefore we decided not to mention them in this paper.

Our second result is the explicit description of the corresponding twisted bicanonical map of these surfaces, and of the singularities of the closure of its image. We can summarize it as follows.
\begin{thm}\label{thm: 2}
Let $S$ be a general surface in $\M$ and let $\eta$ be as in Theorem \ref{thm: 1}. Then $\varphi_{|2K_S +\eta|}$ is a birational map onto a tacnodal surface of degree $10$ in $\PP^3$ with $1-$cycle of degree $26$.
\end{thm}
For the definition of tacnodal surfaces and their $1-$cycle see Definition \ref{def: tacnodal}.

Our main motivation for proving Theorem \ref{thm: 2} is that constructing a similar surface in $\PP^3$ we hope to give in future a positive answer to Question \ref{question 1}. Indeed our last result is the following.

\begin{thm}\label{thm: 3}
Let $\Sigma \subset \PP^3$ be a tacnodal surface of degree $10$ with $1-$cycle of degree $26$. Let $\nu \colon \tilde{\Sigma} \rightarrow \Sigma$ be the normalization map, consider its conductor ideal $\I \subset \OO_\Sigma$ and let $\Sigma_0$ be the open set of $\Sigma$ where $\I$ is invertible. Set $\I^{[2]}$ for the sheaf $j_*(\I^2|_{\Sigma_0})$ where $j$ is the inclusion of $\Sigma_0$ in $\Sigma$.


If \[h^0(\I(6))=0, h^0(\I^{[2]}(11))=1, h^0(\I^{[2]}(12))= 4\] 
then the composition of $\nu$ with the minimal resolution of the singularities of $\tilde{\Sigma}$ is the resolution of the indeterminacy locus of the bicanonical map of a minimal surface of general type with $K^2=3$, $p_g=0$ whose bicanonical system has exactly two base points.

Conversely, let $S$ be a surface of general type with $K^2=3$ whose bicanonical system has exactly two base points, and whose bicanonical map is birational. If the closure of the image of the bicanonical map is tacnodal, it is a surface as above.
\end{thm}

This reduces Question \ref{question 1}, under some reasonable assumptions as the birationality of the bicanonical map, to the problem of constructing a suitable curve in $\PP^3$. We are not able to determine whether such a curve exists yet.

\medskip 

The paper is structured as follows.

In section \ref{family} we describe the family $\M$ constructed in \cite{BFNP13}. 

In section \ref{twistedbicanonical} we prove Theorem \ref{thm: 1}, which we split there in two distinct statements, namely propositions \ref{PROP:2TORSION} and \ref{PROP:2BP}. 

Section \ref{twistedbicanonical map} is devoted to the study of the twisted bicanonical map of these surfaces and to a detailed study of the singular locus of its image, supported on the union of $7$ irreducible curves. In particular we prove Theorem \ref{thm: 2}, which is given by Proposition \ref{PROP:DECIC2} and Proposition \ref{PROP:TAC2}.

In section \ref{towards} we prove Theorem \ref{thm: 3}, by proving Corollary \ref{cor: the one} and Proposition \ref{prop: final}.

\section{The family}\label{family}
We first recall the construction of the  $4-$dimensional family $\M$ of minimal surfaces of general type with $p_g=q=0, K_S^2=3$ and $\pi_1\simeq \Z_4\ltimes \Z_4$ constructed in \cite{BFNP13}.

Consider the product $X$ of $4$ copies of $\PP^1$. This is a Fano fourfold whose automorphisms group is 
$\Aut(X)\simeq \PGL(2)^{4}\rtimes S_4$. 
We consider the automorphisms
\begin{equation}\label{g+h}
g=(\id,A,\id,A)\circ (12)(34)\quad\mbox{ and }\quad h=(\id,A,B,AB) \circ (13)(24)
\end{equation}
where $A$ and $B$ are the automorphisms of $\PP^1$ that are represented,
respectively, by \[(x_0:x_1)\stackrel{A}{\longmapsto}(x_0:-x_1)\quad\mbox{ and }\quad(x_0:x_1)\stackrel{B}{\longmapsto}(x_1:x_0),\] where
$(x_0:x_1)$ are projective coordinates on $\PP^1$.
Both automorphisms have order $4$, and $gh=hg^{3}$; indeed the generated subgroup $G=\spa{g,h} \subset \Aut(X)$ is isomorphic to $\Z_4\ltimes \Z_4$.

We consider $X$ with coordinates $\ux:=(\ux_1,\ux_2,\ux_3,\ux_4)$, where $\ux_i=(x_{i0}:x_{i1})$ are the standard projective coordinates on the $i^{th}$-factor.
In (\cite[4.1]{BFNP13}) it is shown that $G$ acts freely outside of $48$ points. All of them have a cyclic stabilizer of order $2$ or $4$, so all are stabilized by a single element of order $2$ in $G$. Each of the three elements of order $2$ in $G$ stabilizes exactly $16$ points. More precisely: 
\begin{equation}\label{fixed points}
\renewcommand{\arraystretch}{1.35}
\begin{array}{l}
 \Fix(g^2)=\{\ux \quad|\quad \ux_i\in\{(0:1),(1:0)\}\,\},\\
\Fix(h^2)=\{\,\ux \quad|\quad \ux_i\in\{(1:1),(1:-1)\}\,\}\\
\Fix(g^2 h^2)=\{\,\ux \quad|\quad \ux_i\in\{(1:i),(1:-i)\}\,\}.
\end{array}
\end{equation}

The action of $G$ on $X$ induces an action on $|\OO_X(1,1,1,1)|$, which can be lifted to $H^0(\OO_X(1,1,1,1))$, although there is not a natural way to do that. Indeed, as $H^0(\OO_X(1,1,1,1))$ is a set of polynomials in the variables $x_{ij}$, (\ref{g+h}) gives natural liftings $g^*,h^* \colon H^0(\OO_X(1,1,1,1))\rightarrow H^0(\OO_X(1,1,1,1))$. To be more precise, for $a,b,c,d \in \{0,1\}$ we set $a'=1-a,b'=1-b,c'=1-c,d'=1-d$ the respective complements: then the natural liftings would be \[g^*x_{1a}x_{2b}x_{3c}x_{4d}=(-1)^{b+d}x_{1b}x_{2a}x_{3d}x_{4c},\]
\[h^*x_{1a}x_{2b}x_{3c}x_{4d}=(-1)^{b'+d}x_{1c}x_{2d}x_{3a'}x_{4b'}.\]
They have both order $4$, but they do not define a lifting of $G$ as $g^*h^* =- h^*(g^*)^3$.

This can be fixed in $8$ different ways, twisting $g^*$ with the multiplication by $\pm i$ and $h^*$ with the multiplication by $i^d$, where $i \in \C$ is as usual a square root of $-1$. This gives $8$ different liftings of the $G$-action to $H^0(\OO_X(1,1,1,1))$. They are all isomorphic to the regular representation: indeed once chosen one of them, they are all obtained by it by twisting by the $8$ degree $1$ representations of $G$.

We choose the following lifting
\begin{defin}
Define $\rho_1 \colon G \rightarrow \Aut(\HH^0(\OO_X(1,1,1,1)))$ given by the liftings
\begin{eqnarray*}
\rho_1(g) x_{1a}x_{2b}x_{3c}x_{4d}=i(-1)^{b+d}x_{1b}x_{2a}x_{3d}x_{4c},\\
\rho_1(h) x_{1a}x_{2b}x_{3c}x_{4d}=(-1)^{b+d}x_{1c}x_{2d}x_{3a'}x_{4b'}.
\end{eqnarray*}
\end{defin}
The $1-$dimensional invariant subspace  $\HH^0(\OO_X(1,1,1,1))^{G}$ is generated by
\[F_1:=(x_{20}x_{30}-x_{21}x_{31})(x_{11}x_{40}+x_{10}x_{41})-i(x_{20}x_{31}-x_{21}x_{30})(x_{10}x_{40}+x_{11}x_{41}).\]
Its zero locus $V$ is (\cite[Lemma 5.2]{BFNP13}) a singular Fano threefold whose singular locus is contained in $\Fix(h^2)$.

The multiplication map $\mu_{1,1} \colon \HH^0(\OO_X(1,1,1,1))^{\otimes 2} \rightarrow  \HH^0(\OO_X(2,2,2,2))$
is surjective, and therefore there is a unique $G-$action $\rho_2$ on $\HH^0(\OO_X(2,2,2,2))$ such that $\mu_{1,1}$ is  $G-$equivariant. 

We consider the subspace $\HH^0(\OO_X(2,2,2,2))^{-,+}$ of the elements $u$ such that \[\rho_2(g)u=-u\quad \mbox{ and }\quad \rho_2(h)u=u.\] This is the subspace generated by the $Q_i$ in \cite[4.1]{BFNP13}; we give here a different basis, as it will make some formulas in the rest of the paper simpler:
\begin{equation}\label{LISTOFQUADRICS}
\renewcommand{\arraystretch}{1.35}
\begin{array}{ll}
U_0 := &x_{20}x_{21}x_{30}x_{31}(x_{10}^2+x_{11}^2)(x_{40}^2+x_{41}^2)\\&-x_{10}x_{11}x_{40}x_{41}(x_{20}^2+x_{21}^2)(x_{30}^2+x_{31}^2) \\
U_1 :=& x_{10}x_{11}x_{30}x_{31}(x_{20}^2+x_{21}^2)(x_{40}^2+x_{41}^2)\\&+x_{20}x_{21}x_{40}x_{41}(x_{10}^2+x_{11}^2)(x_{30}^2+x_{31}^2) \\
U_2 :=& (x_{10}^2x_{20}^2+x_{11}^2x_{21}^2)(x_{30}^2x_{40}^2+x_{31}^2x_{41}^2) \\
U_3 :=& (x_{11}^2x_{20}^2 +x_{10}^2x_{21}^2)(x_{31}^2 x_{40}^2+ x_{30}^2x_{41}^2) \\
U_4 :=& 4\textstyle\prod_{i=1}^4 x_{i0}x_{i1} \\
U_5 :=& U_0+\frac{1}{2}U_2+\frac{1}{2}U_3+U_4-\frac{1}{2}\prod_{i=1}^4(x_{i0}^2+x_{i1}^2)
\end{array}
\end{equation}
The base points of the induced linear system $\cL\subset |\OO_X(2,2,2,2)|$ are (\cite[Lemma 4.1]{BFNP13}) the following $64$ points: 
\begin{equation}
\label{LISTOFBASEPTS}
\renewcommand{\arraystretch}{1.30}
\begin{array}{l}
\{\ux \quad |\quad \ux_1,\ux_2\in\{(0:1),(1:0)\}, \ux_3\in\{(1:\pm 1)\},\ux_4\in\{(1:\pm i)\} \}\cup \\
\{ \ux \quad |\quad \ux_1,\ux_2\in\{(0:1),(1:0)\}, \ux_3\in\{(1:\pm i)\},  \ux_4\in\{(1:\pm 1)\}  \} \cup \\
\{ \ux \quad| \quad \ux_1\in\{(1:\pm 1)\},  \ux_2\in\{(1:\pm i)\},  \ux_3,\ux_4\in\{(0:1),(1:0)\}  \} \cup \\ 
\{ \ux \quad |\quad \ux_1\in\{(1:\pm i)\},  \ux_2\in\{(1:\pm 1)\},  \ux_3,\ux_4\in\{(0:1),(1:0)\}  \}
\end{array}
\end{equation}
The generic $Y \in \cL$ is  (\cite[Corollary 4.2]{BFNP13}) a smooth Calabi-Yau threefold.
The generic intersection $T=V\cap Y$ is  (\cite[Theorem 5.3]{BFNP13})  a simply connected canonical model of a surface of general type with $p_g=15$, $q=0$ and $K_T^2=48$ which is equipped with a free $G-$action. Its quotient $S=T/G$ is hence the canonical model of a surface of general type with $p_g=q=0$, $K_S^2=3$ and $\pi_1(S)\simeq \Z_4\ltimes \Z_4$. 

We will need later the following

\begin{lem}\label{puredim1}
$\{F_1=U_4=U_2=0\}$ and $\{F_1=U_4=U_3=0\}$ are subschemes of $X$ of pure dimension $1$.
\end{lem}

\begin{proof}

By (\ref{LISTOFQUADRICS}), $\{U_4=U_2=0\}$ has $12$ irreducible components all of dimension $2$:  two of the form $\{x_{1a}=x_{2a'}=0 \}$ ($a\in \{0,1\}$), two of the form $\{x_{3a}=x_{4a'}=0\}$, four of the form $\{x_{na}=x_{30}^2x_{40}^2+x_{31}^2x_{41}^2=0\}$   ($a\in \{0,1\}$ , $n\in \{1,2\}$) and four of the form $\{x_{na}=x_{10}^2x_{20}^2+x_{11}^2x_{21}^2=0\}$   ($a\in \{0,1\}$ , $n\in \{3,4\}$). None of these surfaces is contained in $V$, so $\{F_1=U_4=U_2=0\}$ has pure dimension $1$. 

The proof that $\{F_1=U_4=U_3=0\}$ has pure dimension $1$ is similar.
\end{proof}

\section{A twisted bicanonical system with base points}\label{twistedbicanonical}
\begin{prop}
\label{PROP:2TORSION}
Let $S=T/G$ be an element in $\M$. Then, the system $\cL_{|T}$ is the pull-back of $|2K_S+\eta|$ where $\eta$ is a $2-$torsion element in $\Z/2\Z \times \Z/4\Z \cong \Tors(S) \subset \Pic(S)$ not divisible by $2$.
\end{prop}

\begin{proof}
First of all, recall that the pull-back of holomorphic 2-forms yields a canonical lift of the $G-$ action on $T$ to a representation of $G$ on $H^0(\OO_T(K_T))$, which we will call {\it the canonical representation} of $G$ on $H^0(\OO_T(K_T))$. This is, by the Lefschetz holomorphic fixed point formula, isomorphic to the regular representation minus the trivial representation, say $\C[G]/\C$. Similarly, $\forall d$ the canonical representation of $G$ on $H^0(\OO_T(dK_T))$ is the one induced by imposing the equivariance of the tensor product of sections.

The restriction map $H^0(\OO_X(1,1,1,1)) \rightarrow H^0(\OO_T(1,1,1,1) \cong \OO_T(K_T))$ is surjective with kernel $\langle F_1 \rangle=H^0(\OO_X(1,1,1,1))^G$, so $\rho_1$ induces a lifting of the $G-$ action on $T$ to a representation of $G$ on $H^0(\OO_T(K_T))$ which is also isomorphic to $\C[G]/ \C$. Since two liftings differ by the twist by a character, and only the trivial character preserves $\C[G]/ \C$, $\rho_1$ induces the canonical representation of $G$ on $H^0(\OO_T(K_T))$. Consequently, $\rho_2$ induces the canonical representation of $G$ on  $H^0(\OO_T(2K_T))$.

We denote by $H \subset G$ the subgroup generated by $g^2$ and $h$; note that, for any action of $G$ on  a vector space $V$, the elements in $V^{-,+}$, the space of $h-$invariant vectors such that $g$ acts on them as the multiplication by $-1$, is $H-$invariant and more precisely
$
V^H= V^G \oplus V^{-,+}.
$

Let $S_{H}$ be the quotient $T/H$. 
Then $H^0(\OO_{S_H}(2K_{S_H}))$ pulls back to 
\[H^0(\OO_T(2K_T))^H=H^0(\OO_T(2K_T))^G \oplus \left\langle U_i \right\rangle.\]
On the other hand, the induced map $\pi \colon S_H \rightarrow S$ is an \'etale double cover of $S$, so it determines  a $2-$torsion $\eta \in \Pic(S)$, direct summand of the direct image of $\OO_{S_H}$. The induced involution $S_H \rightarrow S_H$  splits $H^0(\OO_{S_H}(2K_{S_H}))$ as direct sum of  its invariant part, the pull-back of  $H^0(2K_S)$, and its antiinvariant part, pull-back of  $H^0(2K_S + \eta)$, and so $H^0(2K_S + \eta)$ pulls back on $T$ to $\left\langle U_i \right\rangle$.

Finally, let's check that $\eta$ is not divisible by $2$. The only 2-torsion divisible by $2$ is contained in both the cyclic subgroups of $\Tors(S)$ of order $4$, yielding two cyclic covers of degree $4$ of $S$ factoring through $\pi$: these are the quotients of $T$ by the two normal subgroups of $G$  with quotient $\Z/4\Z$, that are $\langle g \rangle$ and $\langle gh^2 \rangle$, which should be then both contained in $H$, forcing $H=\langle g,h^2\rangle$, contradiction.
\end{proof}

\begin{prop}
\label{PROP:2BP}
Let $S$ be as in Proposition \ref{PROP:2TORSION}. The system $|2K_S+\eta|$ has exactly $2$ distinct base points.
\end{prop}

\begin{proof}
Recall that $\cL$ has exactly $64$ base points listed in (\ref{LISTOFBASEPTS}), and $G$ acts on them (as $G$ acts on $\cL$) freely (since this set is disjoint by (\ref{fixed points})), so they are divided in $4$ orbits of $16$ elements. Exactly $32$ of these points (two orbits) belong to $V$, giving the two base points of $|2K_S+\eta|$. 
\end{proof}

The proof of Proposition \ref{PROP:2BP} gives explicitely the two base points of $|2K_S+\eta|$, by giving all elements of the corresponding orbits in $X$.  
They are the orbits of  the points
$((0:1),(0:1),(1:-i),(1:1))$ and $((0:1),(1:0),(1:-i),(1:-1)).$

\section{The induced twisted bicanonical map}\label{twistedbicanonical map}

In this section we consider for every $S\in \M$ the twisted bicanonical linear system $|2K_S+\eta|$ of Proposition \ref{PROP:2TORSION}, which has, by Proposition \ref{PROP:2BP}, exactly $2$ distinct base points, inducing then a twisted bicanonical map $\phi_{|2K_S+\eta|}:S\dashrightarrow \PP^3$ which is not a morphism.

In this section we will prove that $\phi_{|2K_S+\eta|}$ is, for the generic $S\in \M$, birational. Then, we will describe the closure of its image, the twisted bicanonical model $\Sigma$ of $S$, a singular surface of degree $10$ in $\PP^3$.

$\Sigma$ is the closure of the image of $\phi_{|2K_S+\eta|}$.
We have described a natural identification between $\HH^0(\OO_S(2K_S+\eta))$ and the quotient of 
$\HH^0(\OO_X(2,2,2,2))^{-,+}$ by the polynomials vanishing on $T$ (a multiple of $F_1$ and a polynomial defining $Y$).
This gives a relation between $\varphi_{\cL|_T}$ and $\varphi_{|2K_S+\eta|}$, that is  \[\varphi_{\cL|_T}=\varphi_{|2K_S+\eta|}\circ\pi\] where $\pi$ is the projection $T \rightarrow S=T/G$. In particular, the closure of the image of $\varphi_{\cL|_T}$ and the one of $\varphi_{|2K_S+\eta|}$ are the same.

Call $\Sigma_4$ and $\Sigma_3$ respectively the closure of the images of $\varphi_{\cL}$ and $\varphi_{\cL|_V}$. 

\begin{prop}
\label{PROP:DECIC}
$\Sigma_4$ is an irreducible hypersurface of $\PP^5$ of degree $10$.
\end{prop}
\begin{proof} 
An easy direct computation shows that the differential of $\varphi_{\cL}$ at the point $((1:1),(1:0),(1:-1),(1:2))$ has rank $4$. This implies that $\varphi_{\cL}$ is generically finite and then $\Sigma_4$ is an hypersurface of $\PP^5$.

To check its degree, which is the minimal degree on which there is a relation among the $U_i$, we have checked with MAGMA that there aren't relations of degree lower than $10$, and that there exist exactly one relation of degree $10$. The Magma script is in Appendix \ref{MAGMA1}.
\end{proof}
 Our choice of the basis $\left\{U_i\right\}$ induces projective coordinates $u_i$ on  $\PP^5$: these are the coordinates we are going to use in the following. 
 
 \begin{prop}\label{PROP:DECIC2}
$\Sigma_3$ is an irreducible hypersurface of $\PP^4$ of degree $10$, and precisely the hyperplane section $\Sigma_4\cap \{u_5=0\}$ of $\Sigma_4$.
\end{prop}

\begin{proof} Recall that $V$ is defined by $F_1$, so $\Sigma_3$ is contained in the hyperplanes defined by polynomials in $\langle U_i \rangle$ multiple of $F_1$. As $F_1$ is invariant, they form the 1-dimensional subspace $F_1 \otimes H^0(\OO_X(1,1,1,1))^{-,+}$. Since $H^0(\OO_X(1,1,1,1))^{-,+}$ is generated by 
\[F_2=(x_{20}x_{30}-x_{21}x_{31})(x_{11}x_{40}+x_{10}x_{41})+i(x_{20}x_{31}-x_{21}x_{30})(x_{10}x_{40}+x_{11}x_{41})\]
then the only hyperplane containing $\Sigma_3$  is the one corresponding to 
$F_1\otimes F_2 \sim  U_5$ (that's the reason for our choice of $U_5$). This shows $\Sigma_3 \subset \Sigma_4\cap \{u_5=0\}$.

On the other hand, the pull-back of the hyperplane section $\{u_5=0\}$ is the zero locus of $U_5$, the union of $V$ with the zero locus $V_2$ of $F_2$. The involution $\sigma = (\id,\id,B,B)$ of $X$ exchanges $V$ and $V_2$, and $\varphi_{\cL} $ factors through it since all $U_i$ are $\sigma-$invariant: this implies that they have the same image, and therefore $\Sigma_3$ and  $\Sigma_4\cap \{u_5=0\}$ coincide set-theoretically.

$\Sigma_3$ is irreducible as image of an irreducible variety. To conclude it is enough to show that $\Sigma_4\cap \{u_5=0\}$ is reduced. We did it checking that  $\Sigma_4\cap \{u_5=0\}$ has a smooth point, namely $p=(1:x:3:1:0)$  for any root $x$ of the polynomial $11x^4-68x^2+108$.
Indeed,  the polynomial $f$  obtained setting $u_5=0$ in the polynomial computed by the script in Appendix A (we write one expression of it in (\ref{f})) vanishes in $p$, but $\frac{\partial f}{\partial u_0}$ does not.
\end{proof}

The proof of Proposition \ref{PROP:DECIC} shows that the generically finite maps $\varphi_{\cL}$ and its restriction $\varphi_{\cL|_V}$ do not have the same degree. Indeed it is not difficult to show that $\varphi_{\cL|_V}$ is the quotient by an action of $G$, and then it has degree $16$, whereas $\varphi_{\cL}$ has degree $32$, being the quotient by an action of the group of order $32$ generated by $G$ and $\sigma$.

$\Sigma$ is the hyperplane section of $\Sigma_3$ cut by a generic hyperplane $H$ corresponding to the chosen Calabi-Yau variety $Y$ in $(\PP^1)^4$. 

\begin{coroll}\label{COR: birational}
For $S\in \M$ general $\varphi_{|2K_S+\eta|}$ is birational.
\end{coroll}

\begin{proof}
By Proposition \ref{PROP:DECIC2}, $\Sigma$ is an hypersurface of $\PP^3$ of degree $10$. The birationality of $\varphi_{|2K_S+\eta|}$ follows by  $(2K_S)^2=12$ since $|2K_S+\eta|$ has $2$ base points by Proposition \ref{PROP:2BP}.
\end{proof}

The following diagram summarizes the situation.

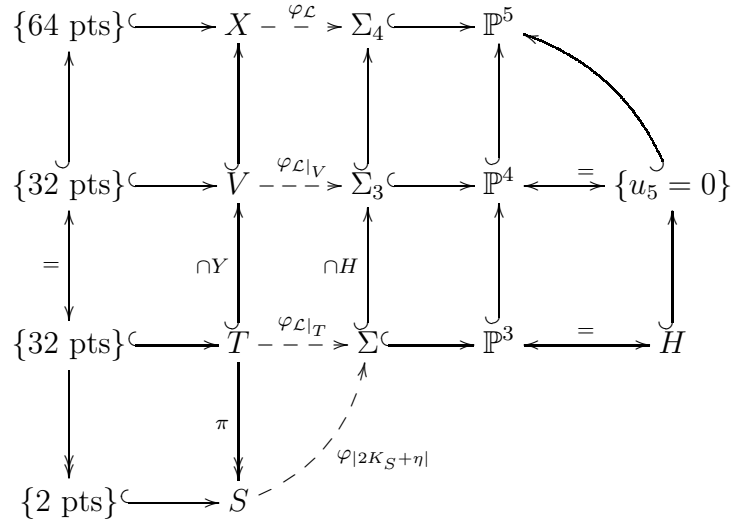
\[\xymatrix@R=40pt@C=30pt{
\{64\mbox{ pts}\}\ar@{^{(}->}[r] & X \ar@{-->}[r]^{\varphi_{\cL}} & \Sigma_4 \ar@{^{(}->}[r] & \PP^5\\
\{32\mbox{ pts}\}\ar@{^{(}->}[r]\ar@{^{(}->}[u] & V \ar@{^{(}->}[u]\ar@{-->}[r]^{\varphi_{\cL|_V}} & \Sigma_3 \ar@{^{(}->}[r]\ar@{^{(}->}[u] & \PP^4\ar@{^{(}->}[u] \ar@{<->}[r]^{=} & \{u_5=0\} \ar@{^{(}->}@/_1pc/[ul]\\
\{32\mbox{ pts}\}\ar@{^{(}->}[r]\ar@{<->}[u]^{=} & T \ar@{^{(}->}[u]^{\cap Y}\ar@{-->}[r]^{\varphi_{\cL|_T}}\ar@{->>}[d]_{\pi} & \Sigma \ar@{^{(}->}[r]\ar@{^{(}->}[u]^{\cap H} & \PP^3\ar@{^{(}->}[u] \ar@{<->}[r]^{=} & H\ar@{^{(}->}[u] \\
\{2\mbox{ pts}\}\ar@{^{(}->}[r]\ar@{<<-}[u] & S \ar@{-->}@/_1pc/[ru]_{\varphi_{|2K_S+\eta|}} 
}\]

The twisted bicanonical model of the generic $S\in \M$ is a general hyperplane section of $\Sigma_3$. In order to study the equation of $\Sigma_3$, a threefold of degree $10$ in $\PP^4$ with coordinates $u_0,\dots, u_4$, let us first introduce the following polynomials:
\begin{table}[H]
\label{TAB:H0DOTSF0}
$\begin{array}{ll}
H_0 :=& u_0\\
H_1 :=& u_1\\
H_2 :=& u_2-u_3\\
H_3 :=& -2u_0+u_2+u_3-2u_4\\
H_4 :=& u_2+u_3+2u_4\\
Q_0 :=& u_4^2-u_2u_3\\
Q_1 :=& (u_0+u_4)^2-u_2u_3\\
G_0 :=& H_1^2H_3+(Q_0-2H_0u_4)H_4+2H_0Q_0 = \\
&(H_0^2 + H_1^2)H_3 + 2H_4Q_0 + (2H_0 - H_4)Q_1\\
G_1 :=& H_1^2H_3+Q_1H_4\\
F_0 :=& H_4G_1+4(3u_3+u_4)H_0H_1^2-2H_0^2H_4^2-H_0H_4^3\\
\end{array}$
\end{table}
A polynomial whose zero-locus is $\Sigma_3$ (obtained by the result of the script in Appendix A setting $u_5=0$) is 
\begin{equation}\label{f}
\begin{array}{rl}
f\hspace{.7em}=\hspace{.7em} Q_1^2 & \hspace{-.7em}\Big((H_0^2+Q_1)^2+Q_1(4H_1^2-(H_3-H_4)(4H_0+3H_4))\Big) + \\
+Q_1G_1 & \hspace{-.7em}\Big(2(Q_0-Q_1)(6H_1^2H_2 - 2H_3H_4^2 + 3H_4^3 - 6H_4Q_0)+ \\ 
 & + 4Q_0H_0^3+4Q_1(4H_0H_1^2+H_0Q_0+3H_4Q_0)-12H_4Q_1^2 + \\ 
 & + F_0(H_3-H_4)+2G_0\big(H_4(H_3-H_4)+2(Q_0-Q_1- H_1^2)\big)\Big) + \\
+4G_1^2 & \hspace{-.7em}\Big(H_0^2Q_0^2\Big).
\end{array}
\end{equation}

The base points of $\cL$ produce two planes (one for each orbit) on $\Sigma_4$, images of the respective exceptional divisors. They are the planes $\Pi_2:=\{u_2=u_4=0\}$ and $\Pi_3:=\{u_3=u_4=0\}$: note that they share a line. 

\begin{prop}
\label{PROP:TAC}
The singular locus of $\Sigma_3$ is the union of $7$ irreducible surfaces, and precisely
\begin{itemize}
\item[] $S_1=\{H_2=H_3=0\}$;
\item[] $S_{2,0}=\{H_0=Q_0=0\}=\{H_0=Q_1=0\}$;
\item[] $S_{2,1}=\{H_1=Q_1=0\}$;
\item[] $S_{3,0}=\{H_0=G_0=0\}=\{H_0=G_1=0\}$;
\item[] $S_{3,1}=\{H_1=G_0=0\}$;
\item[] $S_4=\{H_2=F_0=0\}$;
\item[] $S_6=\{Q_0=G_1=0\}$.
\end{itemize}

$S_1$, $S_{2,0}$ and $S_{2,1}$ are all contained in $\{Q_1=0\}$ and, at a general point of each of these three surfaces, $\Sigma_3$ is locally analytically isomorphic to $\{x_1^2=x_2^4\} \subset \C^4$, with reduced tangent cone that equals the tangent space of $\{Q_1=0\}$ at the same point.

At the general point of the other $4$ components, $\Sigma_3$ is locally analytically isomorphic to $\{x_1^2=x_2^2\} \subset \C^4$.
\end{prop}
\begin{proof}
This is easily computed with the help of MAGMA as follows. 

Computing the ideal of the first derivatives of $f$ and the prime decomposition of its radical, one gets seven prime ideals, that are the ideals of the seven surfaces in the statement, so showing that they are exactly the singular locus of $\Sigma$. 

Then we note that $\{Q_1=H_2=0\}$ is the union of two planes, one of which is $\{H_2=H_3=0\}$, and that at each point of $\{H_2=H_3=0\}$ the tangent plane of $\{Q_1=0\}$ equals $\{H_3=0\}$. Then we check that $f$ belongs to the ideals $(H_2^2,H_3)^2$,  $(H_0^2,Q_1)^2$ and $(H_1^2,Q_1)^2$, that shows that the singularities are as described or worse.

A direct check at a random point of each of the seven components concludes the proof as follows.
\begin{itemize}
\item[-] Each of the surfaces $S_{3,0}$, $S_{3,1}$, $S_4$, $S_6$ contains a point (respectively $(0:1:1:-1:0)$, $(3:0:0:4:4)$, $(2:\frac{1}{\sqrt{5}}:-2:-2:1)$, $(1:-1:1:0:0)$) where the Hessian matrix of $f$ has rank $2$. 
\item[-] Each of the surfaces $S_1$, $S_{2,0}$, $S_{2,1}$ contains a point $p$ (respectively $(1:1:5:5:4)$, $(0:1:4:1:-2)$, $(1:0:4:1:1)$) where the Hessian matrix of $f$ has rank $1$.
\item[-] In each of these last three cases, we blow-up the surface $S_\bullet$ containing $p$, we consider the strict transform $\tilde{\Sigma}_3$ of $\Sigma_3$, and the only point $\tilde{p}$ of $\tilde{\Sigma}_3$ lying over $p$. In a neighbourhood of $\tilde{p}$ it is a complete intersection of two hypersurfaces $A$ and $B$, with $A$ smooth. Using $A$ to eliminate a local coordinate from the polynomial defining $B$, we check that the resulting polynomial is singular at the points lying over $S_\bullet$ with Hessian of rank $2$ at $\tilde{p}$.
\end{itemize}
\end{proof}

\begin{remark}
The singular locus of $\Sigma_3$ determines $\Sigma_3$. Indeed a direct computation with MAGMA shows that there is a unique threefold of $\PP^4$ of degree $10$ which is singular along $\bigcup S_\bullet$, and which as a tacnode at the general point of $S_1$ and of each $S_{2,j}$ with reduced tangent cone that equals the tangent space of $Q_1$ at the same point.
\end{remark}

\begin{remark}\label{REM: singular sections}
Note that the first index of each surface equals its degree, and the second index, when it exists, distinguishes two surfaces of the same degree by the index of the equation of the hyperplane containing it.
The three hyperplane sections cut by $H_0$, $H_1$ and $H_2$ are sections of double points. Indeed there are a few useful hypersurface sections which are supported on these surfaces, and namely:
\begin{itemize}
\item[$H_0$:] cuts $2(S_{2,0}+S_{3,0})$;
\item[$H_1$:] cuts $2(S_{2,1}+S_{3,1})$;
\item[$H_2$:] cuts $2(S_1+S_4)$;
\item[$Q_0$:] cuts $2(S_{2,0}+S_6+\Pi_2+\Pi_3)$ where $\Pi_i=\{u_i=u_4\}=0$ are the images of the exceptional divisors of the base points of $\cL_{|V}$;
\item[$Q_1$:] cuts $4(S_1+S_{2,0}+S_{2,1})$;
\item[$G_1$:] cuts $2(2S_1+2S_{2,1}+S_{3,0}+S_6)$.
\end{itemize}
It is worth mentioning that $S_{3,1}$, $S_4$ and $S_6$ are singular along the same line $\{H_1=H_2=u_2+u_3+2u_4=0\}$, which is a double line for all of them. The surfaces $S_\bullet$ have no other singularities in codimension $1$.
\end{remark}

We deduce a description of the singularities of $\Sigma$.
\begin{defin}\label{def: tacnodal}
A surface $\Sigma \subset\PP^3$ is said to be \textbf{tacnodal} with $1-$cycle $\sum_i m_i C_i$ if it is irreducible and 
 \begin{itemize}
\item all $C_i$ are irreducible and reduced curves in $\PP^3$;
\item $\forall i$, $m_i \in \N$;
\item the non normal locus of $\Sigma$ is the union of the $C_i$;
\item at the general point $P\in C_i$, $\Sigma$ is locally analytically of the form $x^2=y^{2m_i}$
for suitable local analytic coordinates $(x,y,z)$;
\item The normalization $\nu : \ti{\Sigma}\rightarrow \Sigma$ has only canonical singularities.
\end{itemize}
\end{defin}

Let  then $S\in \M$ be general, and let $\Sigma$ be its twisted bicanonical model relative to $|2K_S+\eta|$, so 
$\Sigma = H \cap \Sigma_3$ for a general hyperplane section $H\cong\PP^3 \subset \PP^4$.

\begin{prop}
\label{PROP:TAC2}
 $\Sigma$ is tacnodal with $1-$cycle 
\[2C_1+2C_{2,0}+2C_{2,1}+C_{3,0}+C_{3,1}+C_4+C_6\]
where $C_{\bullet}=S_{\bullet} \cap H$, and the reduced tangent cone at the general point $P$ of $C_1$, $C_{2,0}$ and $C_{2,1}$ equals the tangent plane of $\{Q_1=0\} \cap H$.
\end{prop}

\begin{proof}
As $\Sigma$ is a general hyperplane section of $\Sigma_3$, the only claim which does not follow immediately from Proposition \ref{PROP:TAC} is that the normalization $\nu : \ti{\Sigma}\rightarrow \Sigma$ has only canonical singularities. 
\vspace{2mm}

\noindent 
The twisted bicanonical map $\varphi_{|2K_S+\eta|}$ is the composition of the inverse of the blow up $p:\ti{S}\rightarrow S$ at the two base points of $|2K_S+\eta|$
with a birational (by Corollary \ref{COR: birational}) morphism $\psi:\ti{S}\rightarrow \Sigma$. Being $\ti{S}$ smooth (and hence normal), $\psi$ factors as composition of $\nu$ with a birational morphism $\ti{\varphi} \colon \ti{S}\rightarrow \ti{\Sigma}$ as in the following diagram:

\[\xymatrix@R=40pt@C=40pt{
\ti{S}\ar@{->}[r]^{\ti{\varphi}}\ar@{->}[d]_{p}\ar@{->}[rd]_{\psi} & \ti{\Sigma} \ar@{->}[d]^{\nu} \\
S\ar@{-->}[r]_{\varphi_{|2K_S+\eta|}} & \Sigma
}\]

Since the only (by the uniqueness of the minimal model) exceptional divisors of the first kind of $\tilde{S}$ are not contracted by $\ti{\varphi}$, $\ti{\varphi}$ is the minimal resolution of the singularities of $\ti{\Sigma}$. So $K_{\ti{S}}=\ti{\varphi}^*K_{\ti{\Sigma}}-\E$ for some effective divisor $\E$: we will conclude the proof by showing that $\E$ is trivial. 
\vspace{2mm}

\noindent Setting $D_\bullet$ for the reduced trasform of $C_\bullet$ on $\ti{\Sigma}$, the conductor divisor of $\nu$ is 
\[D:=2D_1+2D_{2,0}+2D_{2,1}+D_{3,0}+D_{3,1}+D_4+D_6\] and then, setting $H$ for the pull-back of a hyperplane section of $\Sigma$, a canonical divisor for $\ti{\Sigma}$ is \begin{equation}\label{KSigmatilde}
K_{\ti{\Sigma}}=6H-D.
\end{equation}

By Remark \ref{REM: singular sections}, the pull back of the quintic $H_0H_1H_2Q_0$ to $\ti{\Sigma}$ shows  $5H \equiv D_1+2D_{2,0}+D_{2,1}+D_{3,0}+D_{3,1}+D_4+D_6+2(\ti{\Pi}_2+\ti{\Pi}_3)$, where we denote by $\ti{\Pi}_i$ the pull-back of the line $\Pi_i \cap H$. So
\begin{equation}\label{eq:Ksigma}
K_{\ti{\Sigma}}\equiv H-D_1-D_{2,1}+2(\ti{\Pi}_2+\ti{\Pi}_3)
\end{equation}
Similarly,  the pull back of $H_0H_1^2H_2^2Q_0Q_1G_1$ to $\ti{\Sigma}$ shows $12H\equiv 2D+2(D_1+D_{2,1}+\ti{\Pi}_2+\ti{\Pi}_3)$ and then, by (\ref{KSigmatilde}), $2(D_1+D_{2,1}+\ti{\Pi}_2+\ti{\Pi}_3)$ is an effective divisor in $|2K_{\tilde{\Sigma}}|$, so there is a 2-torsion divisor $\ti{\eta}$ such that 
\begin{equation}\label{eq:Ksigma+eta}
K_{\ti{\Sigma}}+\ti{\eta} \equiv D_1+D_{2,1}+\ti{\Pi}_2+\ti{\Pi}_3
\end{equation}
Summing (\ref{eq:Ksigma}) and (\ref{eq:Ksigma+eta}) we get
\begin{equation*}
2K_{\ti{\Sigma}}+\ti{\eta} \equiv H+3\ti{\Pi}_2+3\ti{\Pi}_3
\end{equation*}
Pulling back to $\tilde{S}$ and passing to numerical equivalence $\equiv_{num}$ we obtain 
\[2K_{\ti{S}}+2\E \equiv_{num} 2K_{\ti{S}} -3E_2-3E_3+3\ti{\varphi}^* \ti{\Pi}_2+3\ti{\varphi}^* \ti{\Pi}_3\]
where $E_i$ is the exceptional divisor of $\ti{S}\rightarrow S$ mapped to $\ti{\Pi}_i$. That shows 
\[2\E \equiv_{num} 3(\ti{\varphi}^* \ti{\Pi}_2-E_2)+3(\ti{\varphi}^* \ti{\Pi}_3-E_3).\]

This shows that if $\tilde{\Sigma}$ has a singular point which is not canonical, it lies over a singular point of $\Sigma$ belonging to $\Pi_2$ or $\Pi_3$; so there is a curve $C \subset S$ (image of a component of $\E$) contracted by $\varphi_{|2K_S+\eta|}$ to a point of $\Pi_2$ or $\Pi_3$. Since $T$ is defined by a general $Y$, there is an effective divisor $Z \subset V$ which is contained in either $\{U_2=U_4=0\}$ or $\{U_3=U_4=0\}$ (which are, respectively, the preimages of $\Pi_2$ and $\Pi_3$). But this is impossible, as $\{F_1=U_4=U_j=0\} \subset X$, for $j\in \{2,3\}$, has pure dimension $1$, as shown in Lemma \ref{puredim1}.
\end{proof}

\begin{remark}
In the proof of Proposition \ref{PROP:TAC2} we use the polynomial $H_0H_1^2H_2^2Q_0Q_1G_1$ to show that $D_1+D_{2,1}+\ti{\Pi}_2+\ti{\Pi}_3$ is one of the three effective divisors on $\ti{\Sigma}$ that differs from a canonical divisor by a $2$-torsion. Analogous argument using instead $H_0^2H_1^2H_2^2Q_0^2Q_1$ and $H_0^3H_1^2H_2^2Q_0G_1$ shows that the other two divisors with the same property are respectively $D_{2,0}+2(\ti{\Pi}_2+\ti{\Pi}_3)$ and $D_{3,0}+\ti{\Pi}_2+\ti{\Pi}_3$. So the $2$-torsion bundles are obtained by taking the difference of two of these three divisors.
\end{remark}

\section{Towards the construction of surfaces whose bicanonical systems has base points}\label{towards}
Proposition \ref{PROP:TAC2} suggests a possible construction of surfaces with $p_g=0$, $K^2=3$, bicanonical system with two base points and birational bicanonical map. The idea is to construct its bicanonical image assuming it is a tacnodal surface in $\PP^3$.

We start with the following 
\begin{lem}\label{lem: h0(I(11))>0}
Let $\Sigma \subset \PP^3$ be a tacnodal surface of degree $10$, let $\nu \colon \tilde{\Sigma} \rightarrow \Sigma$ be the normalization map and consider its conductor ideal $\I \subset \OO_\Sigma$. Assume, moreover,
\[h^0(\I(6))=0\quad \mbox{ and }\quad h^0(\I^{[2]}(11))>0.\] 

Then $\tilde{\Sigma}$ is a surface of general type with $p_g=0$ and $\nu$ is the map induced by a subsystem of $|2K_{\tilde{\Sigma}}|$. 
\end{lem}

Note that the surfaces  of the previous section verify almost all these conditions: for them $h^0(\I^{[2]}(11))=0$: this is the condition we need to relate $\nu$ with the bicanonical system.  

\begin{proof}
By assumption we have
\begin{equation*}
\nu_*(\omega_{\ti{\Sigma}})=\cI\omega_{\Sigma}=\cI(-4+10)= \cI(6)
\end{equation*}
so that
\[p_g(\ti{\Sigma})=\h^0(\omega_{\ti{\Sigma}})=\h^0(\nu_*(\omega_{\ti{\Sigma}}))=\h^0(\cI(6))=0.\]

Let $\sum m_i C_i$ be the $1-$cycle of $\Sigma$, and let $D_i$ be the preimage $\nu^{-1}C_i$ with the reduced structure. Set $H$ for the pull-back to $\tilde{\Sigma}$ of a hyperplane.
 By adjunction for finite mappings a canonical divisor of $\tilde{\Sigma}$ is $K_{\tilde{\Sigma}}=6H-\sum m_iD_i$.
\vspace{2mm}

\noindent Since $\nu_* \omega_{\ti{\Sigma}}^2=\I^{[2]}(12)$ and, by assumption, $h^0(\I^{[2]}(11))>0$, there is an effective divisor $\EE$ in \[|11H-2\sum m_i D_i|=|2K_{\tilde{\Sigma}}-H|.\]  
We deduce that $H<2K_{\tilde{\Sigma}}$ and then $\nu$ is defined by a subsystem of the bicanonical system. Moreover $\forall d$, $dH<2dK_{\tilde{\Sigma}}$, that ensures that $h^0(2dK_{\tilde{\Sigma}}) \geq h^0(\OO_{\Sigma}(d))$ grows as a polynomial of degree $2$ in $d$, and then $\tilde{\Sigma}$ is of general type.
\end{proof}

To get the surfaces we are looking for, we need to fix the degree of the $1-$cycle to the same value of the surfaces in the previous section, namely $26$.

\begin{prop}\label{prop: list}
Let $\Sigma$ be as in Lemma \ref{lem: h0(I(11))>0} and assume moreover that its $1-$cycle has degree $26$. Then one of the following occurs
\begin{itemize}
\item[i)] The composition of $\nu$ with the minimal resolution of the singularities of $\tilde{\Sigma}$ is the resolution of the indeterminacy locus of the bicanonical map of a minimal surface of general type with $K^2=3$, $p_g=0$ whose bicanonical system has exactly two base points (possibly infinitely near). 
\item[ii)] The minimal resolution of the singularities of $\tilde{\Sigma}$ is the blow up in a point of a minimal surface of general type with $p_g=0$, $K^2=4$. 
\item[iii)] $\tilde{\Sigma}$ is a surface of general type with nef canonical class, $p_g=0$ and $4 \leq K^2 \leq 6$.
\end{itemize}
\end{prop}

\begin{proof} By assumption $\tilde{\Sigma}$ has only canonical singularities. We consider the minimal resolution of its singularities $\tilde{S}$. The canonical system of $\tilde{S}$ equals the pull-back of the canonical system of $\tilde{\Sigma}$. We denote here by $H$ the pull-back on $\tilde{S}$ of a hyperplane section of $\Sigma$, and by $D_i$ the pull-back on $\tilde{S}$ of the irreducible component of the conductor divisor mapping to $C_i$.

By the proof of Lemma \ref{lem: h0(I(11))>0} there is an effective divisor $\EE$ in $|2K_{\tilde{S}}-H|$. We have
\[H\EE = H\left( 11H-2\sum m_i D_i \right)=11H^2-2H\left( \sum m_iD_i \right)= 11 \deg \Sigma - 4 \deg C\]

which, under our assumptions, gives us
\begin{equation}
  H\EE = 110-104=6.
\end{equation}

By the Hodge index theorem \[10\EE^2=H^2\EE^2 \leq (H\EE)^2=36 \Rightarrow \EE^2 \leq 3\] 
and by the genus formula $K_{\tilde{S}}\EE+\EE^2=\frac{H\EE+3\EE^2}{2}$ is even so $\EE^2$ is of the form $4k+2$ for some integer $k$. 
\vspace{2mm}

If $K_{\tilde{S}}$ is nef, then by $0 \leq 2K_{\tilde{S}}\EE=6+\EE^2$ we deduce $\EE^2\in\{-6, -2, 2\}$. Then 
\[K_{\tilde{S}}^2=\frac{(H+\EE)^2}{4}=\frac{22+\EE^2}{4}\]
and we are in case iii).

Else $K_{\ti{S}}$ is not nef, and then there is a rational curve $E_1$ in $\ti{S}$ with $E_1^2=-1$, which implies that $2E_1$ is contained in the fixed part of $|2K_{\tilde{S}}|$, so $2E_1 \leq \EE$. 
\vspace{2mm}

Since $H$ is nef and ample on the curves not orthogonal to $K_{\tilde{S}}$, then $HE_1>0$ and therefore $(\EE-2E_1)E_1=(2(K_{\tilde{S}}-E_1)-H)E_1=-HE_1< 0$, so $3E_1 \leq \EE$. Moreover, the inequality is strict since $(H+3E_1)^2$ is odd whereas $(H+\EE)^2=4K_{\tilde{S}}^2$ is even. We can write 
\[\EE=3E_1+\EE_1\]
with $\EE_1$ effective. 
Since $H$ is nef, $H E_1>0 $, and $H\EE=6$, then $(H\EE_1,HE_1)=(3,1)$ or $(0,2)$. In the latter case $H\EE_1=0$ and then (since $H$ is ample on the curves not orthogonal to $K_{\tilde{S}}$) also $K_{\tilde{S}}\EE_1=0$, so  $\EE_1$ is union of $-2$-curves. Then $\EE_1^2$ is even. On the other hand  $\EE_1E_1=(2K_{\tilde{S}}-H-3E_1)E_1=-2-2+3=-1$ and hence we have $\EE_1^2=\EE_1^2-2K_{\tilde{S}}\EE_1=\EE_1^2-(H+\EE_1+3E_1)\EE_1=-(H+3E_1)\EE_1=3$, a contradiction. So

\begin{equation}\label{HE1}
H\EE_1=3\quad \quad HE_1=1
\end{equation}
\begin{equation}\label{E1E1}
\EE_1E_1=(2K_{\tilde{S}}-H-3E_1)E_1=-2-1+3=0
\end{equation}
\begin{equation}\label{2KE1}
2K_{\tilde{S}}\EE_1=(H+\EE_1+3E_1)\EE_1=3+\EE_1^2.
\end{equation}

By the Hodge index theorem $10\EE_1^2=H^2\EE_1^2 \leq (H\EE_1)^2 = 9 \Rightarrow \EE_1^2 \leq 0$. 
By the genus formula, $\EE_1^2+K_{\tilde{S}}\EE_1=\frac{3}{2}(\EE_1^2+1)$ is even, so $\EE_1^2$ is of the form $-(4k+1)$, $k\in \mathbb{N}$.
\vspace{2mm}

\noindent If $K_{\tilde{S}}-E_1$ is nef, 
\[0\leq (K_{\tilde{S}}-E_1)\EE_1=\frac{3+\EE_1^2}{2},\] 
so, we obtain $\EE_1^2=-1$, $\EE^2=\EE_1^2+9E_1^2=-10$ and $K_{\tilde{S}}^2=\frac{22-10}{4}=3$. So contracting $E_1$ we get a surface with nef canonical system and $K^2=4$: this is case ii).
\vspace{2mm}

Else $K_{\tilde{S}}-E_1$ is not nef. Then the surface $\tilde{S}_1$ obtained by $\tilde{S}$ contracting $E_1$ is not minimal, as its canonical class $K_{\tilde{S}_1}$ is not nef. So we have a sequence of two 
elementary contractions $\tilde{S} \rightarrow \tilde{S}_1 \rightarrow \tilde{S}_2$.

Then $\tilde{S}$ is obtained by $\tilde{S}_2$ by blowing up two points, possibly infinitely near. In both case there is a further effective divisor $E_2$ with 
\begin{equation}\label{E2}
E_1E_2=0,\quad E_2^2=K_{\tilde{S}}E_2=-1:
\end{equation}
if the two points are not infinitely near then $E_2$ is irreducible, else there is an irreducible curve $E'_2$  with $K_{\tilde{S}}E'_2=0$, and we set $E_2=E_1+E_2'$. Note $E_1E_2'=1$, $E_2E_2'=-1$,

\vspace{2mm}

We claim that in both cases $3E_2 \leq \EE_1$.

\vspace{2mm}

Indeed, if $E_2$ is irreducible, the same argument used to show  $3E_1 \leq \EE$ shows also $3E_2 \leq \EE$, and then, since $E_1$ and $E_2$ are distinct irreducible curves, $3E_2 \leq \EE_1$.

\vspace{2mm}

In the latter case, $E_2=E_1+E_2'$. The exceptional divisor $E_1+E_2$ of the "double" contraction  $\tilde{S}\rightarrow \tilde{S}_2$ is in the fixed part of $|2K_{\tilde{S}}|$, and so $2E_1+2E_2 \leq \EE$. 
By (\ref{E1E1}) $\EE E_1= 3E_1^2+\EE_1E_1=-3$ then $(\EE-2E_1-2E_2)E_1=-1<0$, so, since $E_1$ is irreducible, $3E_1+2E_2 \leq \EE$.

Then, since $(\EE-3E_1-2E_2)E_2'=(2K_{\tilde{S}}-H)E_2'-(3E_1+2E_2)E_2'=-HE_2'-1<0$ then $3E_1+2E_2+E'_2=2E_1+3E_2 \leq \EE$. Once more, since  $(\EE-2E_1-3E_2)E_1=-1$, we conclude $3E_1+3E_2 \leq \EE$ and so $3E_2 \leq \EE_1$.

\vspace{2mm}

Then in all cases $3E_2 \leq \EE_1$. 
Since $E_2$ is not orthogonal to $K_{\tilde{S}}$, $HE_2 >0$, and by (\ref{HE1}) $H\EE_1=3$, we obtain $HE_2=1$ and 
\begin{equation*}
\EE=3(E_1+E_2)+Z.
\end{equation*}
where $Z$ is an effective divisor orthogonal to $H$ and then to $K_{\tilde{S}}$ and to $\EE$.

By (\ref{E1E1}) and (\ref{E2}) it follows $ZE_1=0$. By (\ref{2KE1}) $-6=6K_{\tilde{S}}E_2=2K_{\tilde{S}}\EE_1=3+(3E_2+Z)^2=-6+6E_2Z+Z^2$, and then 
\[Z^2=-6E_2Z.\]
On the other hand  $Z^2=Z(Z-\EE)=-3ZE_2$. Then $Z^2=ZH=0$ that implies immediately $Z=0$. 

\noindent So
\begin{equation*}
\EE=3(E_1+E_2),\quad K_{\tilde{S}}^2=\frac{22-18}{4}=1,
\end{equation*}
and, by contracting both $E_1$ and $E_2$ we get a surface $S$ with $K_S^2=3$ which is minimal as its bicanonical system has no fixed components. Then $h^0(2K_S)=4$ and therefore $\nu$ is the resolution of the indeterminacy locus of its bicanonical map. 
\end{proof}

\begin{remark}
The cases of Proposition \ref{prop: list} are almost completely distinguished by $h^0(\I^{[2]}(12))$ as it equals $h^0(2K_{\tilde{\Sigma}})$ which equals, by the Riemann-Roch theorem, the canonical degree of the minimal model plus $1$. So, in case i) we will have $h^0(\I^{[2]}(12))=4$, in case ii)  $h^0(\I^{[2]}(12))=5$, in case iii)  $h^0(\I^{[2]}(12))\in\{5,6,7\}$. We are mainly interested in case i) of Proposition \ref{prop: list}, the only case in which $h^0(\I^{[2]}(12))$ equals $4$. 
\end{remark}

It follows

\begin{coroll}\label{cor: the one}
Let $\Sigma \subset \PP^3$ be a tacnodal surface of degree $10$ with $1-$cycle of degree $26$. Let $\nu \colon \tilde{\Sigma} \rightarrow \Sigma$ be the normalization map and consider its conductor ideal $\I \subset \OO_\Sigma$. If one assumes 
\[h^0(\I(6))=0, h^0(\I^{[2]}(11))=1, h^0(\I^{[2]}(12))= 4\] 
then the composition of $\nu$ with the minimal resolution of the singularities of $\tilde{\Sigma}$ is the resolution of the indeterminacy locus of the bicanonical map of a minimal surface with $K^2=3$, $p_g=0$ whose bicanonical system has two base points.
\end{coroll}

In other words, if we find a curve, a scheme of pure dimension $1$, $C$ of degree $26$ in $\PP^3$ whose ideal sheaf $\J$ has $h^0(\J(6))=0$  but such that $H_*^0(\J^{[2]}):=\oplus H^0(\J^{[2]}(d))$ is minimally generated, as ideal of the polynomial ring in 4 variables, by polynomials $f_1, \ldots, f_r$ with $\deg f_1=10$,  $\deg f_2=11$, $\forall i \geq 3$ $\deg f_i \geq 13$, then we set $\Sigma=\{f_1=0\}$ (which is singular along $C$). If $\Sigma$ is tacnodal with $1-$cycle $C$ (in other words the Zariski tangent space of $C$ at a general point of each component is at most $2$ and $\Sigma$ {\it has no further singularities}), then we have found a minimal surface of general type with $p_g=0$, $K^2=3$ and bicanonical system with two base points.
\vspace{2mm}

\noindent We can't prove that such a curve $C$ exists but we can prove the following.

\begin{prop}\label{prop: final}
Let $S$ be a minimal surface of general type with $p_g=0$, $K^2_S=3$ and whose bicanonical system has exactly two base points.  Let $\Sigma$ be the closure of the image of the bicanonical map. Assume that the bicanonical map is birational. Then, if $\Sigma$ is tacnodal, then it is a surface as in Corollary \ref{cor: the one}.
\end{prop}

\begin{proof}
Let $\tilde{S}$ be the blow up of $S$ at the two base points and let $H$ be the pull-back on $\tilde{S}$ of a hyperplane of $\PP^3$. Then $\deg \Sigma= H^2=(2K_S)^2-2=10$.

The bicanonical map $\tilde{S} \rightarrow \Sigma$ factors through the normalization $\nu \colon \tilde{\Sigma}\rightarrow \Sigma$ and, since we assumed $\Sigma$ tacnodal, the induced map $\tilde{S}\rightarrow \tilde{\Sigma}$ is the minimal resolution of the singularities of $\tilde{\Sigma}$, which are by assumption canonical singularities.

If $\I$ is the conductor ideal of $\nu$ we have $h^0(\I(6))=p_g(\tilde{\Sigma})=p_g(\tilde{S})=p_g(S)=0$.
\vspace{2mm}

\noindent By assumption $H+3E_1+3E_2$ is a bicanonical divisor, so $0 \neq h^0(2K_{\tilde{\Sigma}}-H)=h^0(\I^{[2]}(11))$: all assumptions of Lemma \ref{lem: h0(I(11))>0} are fulfilled. Let $C$ be the $1-$cycle of $\Sigma$, $D \in \tilde{\Sigma}$ the conductor divisor. Then $\deg C= \frac12 HD= \frac12 H(6H-K)= \frac14 H(11H-3E_1-3E_2)=\frac{110-6}4=26$. So we are in the assumptions of Proposition \ref{prop: list} and, since $K_{S}^2=3$, we are in case iii).
\end{proof}


\appendix
\section{Magma Code}
\label{MAGMA1}

The following source is a Magma code used to determine the only non-trivial relation of degree lower that or equal to $10$ between the elements of a basis of $\HH^0(X,\OO_X(2,2,2,2))^{-,+}$.

\begin{center}
\begin{codice_magma}[caption={A Magma code to determine the algebraic relations a given degree between some objects.}]
// Given a sequence L, it returns the standard generators for the n-th
// symmetric power of the space generated by the elements of L.
function SymP(L,n) 
    if (n eq 0) then 
        return [1];
    elif (n eq 1) then
        return L;
    elif (n lt 0) then
        return [0];
    else 
        if (#L gt 1) then
            return [L[1]*x : x in SymP(L,n-1)] cat
			   SymP(Exclude(L, L[1]),n);
        else 
            return [L[1]^n];
        end if;
    end if;
end function;

// Given a sequence Q of elements, it returns a basis 
// for the space of algebraic relations of degree n 
// between the elements of Q.
function Rels(Q,n)
    TS:=SymP(Q,n);
    TB:=[];
    for q in TS do
        TB cat:= Monomials(q); 
    end for;
    TB:=Setseq(Seqset(TB));
    TA:=Matrix([[MonomialCoefficient(TS[j],TB[i]) :
			 j in [1..#TS]] : i in [1..#TB]]);
    return Basis(Kernel(Transpose(TA)));
end function; 
//----------------------------------------------------------

K:=Rationals();
K8<x10,x11,x20,x21,x30,x31,x40,x41>:=PolynomialRing(K,8);

U0:=x20*x21*x30*x31*(x10^2+x11^2)*(x40^2+x41^2)-
	x10*x11*x40*x41*(x20^2+x21^2)*(x30^2+x31^2);
U1:=x10*x11*x30*x31*(x20^2+x21^2)*(x40^2+x41^2)+
	x20*x21*x40*x41*(x10^2+x11^2)*(x30^2+x31^2);
U3:=(x11^2*x20^2+x10^2*x21^2)*(x31^2*x40^2+x30^2*x41^2);
U2:=(x10^2*x20^2+x11^2*x21^2)*(x30^2*x40^2+x31^2*x41^2);
U4:=4*x10*x11*x20*x21*x30*x31*x40*x41;
U5:=U0+(1/2)*(U2+U3)+U4+
	(1/2)*(x10^2+x11^2)*(x20^2+x21^2)*(x30^2+x31^2)*(x40^2+x41^2);

// A basis for H^0(X,O_X(2,2,2,2))^{-,+}.
U:=[U0,U1,U2,U3,U4,U5];
Km<u0,u2,u2,u3,u4,u5>:=PolynomialRing(K,6);
Us:=[Km.i : i in [1..6]];

// This allows to check that there are no relations of degree 
// lower than 10 between the elements of U.
time #Rels(U,9); 

// A basis for the space of relation of degree 10 between
// the elements of U.
time Rn:=Rels(U,10); // It has only one element.

// The relation(s) written with respect to the coordinates ui.
TU:=SymP(Us,10);
[&+([TU[i]*b[i] : i in [1..#TU]]) : b in Rn];
Sols:=[&+([TU[i]*b[i] : i in [1..#TU]]) : b in Rn];
// The polynomial of degree 10 which is satisfied by U0,....,U5
DecicPol:=Sols[1];
\end{codice_magma}
\end{center}


\begin{thebibliography}{AAAA11}

\bibitem[BCP11]{survey}
I. Bauer, F. Catanese, R, Pignatelli, 
{\it Surfaces of general type with geometric genus zero: a survey}. 
Complex and differential geometry, 1--48, Springer Proc. Math., {\bf 8}, Springer, Heidelberg, 2011.

\bibitem[Bea88]{bea88}
A. Beauville, 
{\it Annulation du $H^1$ et syst\'emes paracanoniques sur les surfaces}.
J. Reine Angew. Math. {\bf 388} (1988), 149--157.


\bibitem[Bea99]{beauville}
A. Beauville, 
{\it A Calabi-Yau threefold with non-abelian fundamental group}. 
New trends in algebraic geometry (Warwick, 1996), 13--17, 
London Math. Soc. Lecture Note Ser., {\bf 264}, Cambridge Univ. Press, Cambridge, 1999. 

\bibitem[BFNP14]{BFNP13}
G. Bini, F.F. Favale, J. Neves, R. Pignatelli,
{\it New examples of Calabi--Yau threefolds and genus zero surfaces}.
Communications in Contemporary Mathematics  {\bf 16} (2014), no. 02, 1350010, 20 pp.

\bibitem[Bom73]{bombieri}
E. Bombieri,
{\it Canonical models of surfaces of general type}. 
Inst. Hautes \'Etudes Sci. Publ. Math. No. {\bf 42} (1973), 171--219.

\bibitem[CC91]{cirofab}
F. Catanese, C. Ciliberto, 
{\it Surfaces with $p_g=q=1$}. 
Problems in the theory of surfaces and their classification (Cortona, 1988), 49--79, 
Sympos. Math., XXXII, Academic Press, London, 1991.
 
\bibitem[CP06]{pencils}
F. Catanese, R. Pignatelli
{\it Fibrations of low genus. I}. 
Ann. Sci. \'Ecole Norm. Sup. (4) {\bf 39} (2006), no. 6, 1011--1049.

\bibitem[CP15]{pirola2}
A. Castorena, G.P. Pirola,
{\it Some results on deformations of sections of vector bundles}
arXiv:1510.02964
    
\bibitem[GL87]{gl}
M. Green, R. Lazarsfeld, 
{\it Deformation theory, generic vanishing theorems, and some conjectures of Enriques, Catanese and Beauville}.
Invent. Math. {\bf 90} (1987), no. 2, 389--407.    
 
\bibitem[HM06]{hacon}
 C.D. Hacon, J. McKernan, 
 {\it Boundedness of pluricanonical maps of varieties of general type}. 
 Invent. Math. {\bf 166} (2006), no. 1, 1--25. 
 
 \bibitem[MP08]{mendes lopes pardini}
M. Mendes Lopes, R. Pardini
{\it Numerical Campedelli surfaces with fundamental group of order 9}. 
J. Eur. Math. Soc. (JEMS) {\bf 10} (2008), no. 2, 457--476.
  
 \bibitem[MPP15]{pirola1}
M. Mendes Lopes, R. Pardini, G.P. Pirola,
{\it Continuous families of divisors, paracanonical systems and a new inequality for varieties of maximal Albanese dimension}. 
Geom. Topol. {\bf 17} (2013), no. 2, 1205--1223. 
 
\bibitem[Rei88]{reider}
I. Reider, 
{\it Vector bundles of rank 2 and linear systems on algebraic surfaces}. 
Ann. of Math. (2) {\bf 127} (1988), no. 2, 309--316
\end{thebibliography}
\end{document}